\documentclass{amsart}
\usepackage{amsfonts}
\usepackage{amsmath,amscd}
\usepackage{amssymb}
\usepackage{amsthm}
\usepackage{newlfont}

 \newtheorem{thm}{Theorem}[section]
 
 \newtheorem{lem}[thm]{Lemma}
 
 \theoremstyle{definition}
 
 \theoremstyle{remark}

 \numberwithin{equation}{section}

\begin{document}

\title[A remark on the capability of finite $p$-groups]
 {A remark on the capability of finite $p$-groups}

\author[P. Niroomand]{Peyman Niroomand}
\author{Mohsen Parvizi}
\address{School of Mathematics and Computer Science\\
Damghan University of Basic Sciences, Damghan, Iran}
\email{p$\_$niroomand@yahoo.com} \email{parvizi@dubs.ac.ir}




\date{\today}


\begin{abstract}
In this paper we classify all capable finite $p$-groups with derived
subgroup of order $p$ and $G/G'$ of rank $n-1$.
\end{abstract}

\maketitle

\section{Motivation}
Recall the famous question of P. Hall about a given group $G$ ``Can
we decide that $G\cong H/Z(H)$ for a group $H$?" That is an
interesting question but unfortunately finding necessary and
sufficient conditions for a group $G$ to be isomorphic to $H/Z(H)$
is not easy. If it is possible such group is called capable
following to \cite{ha}. It is known that which of finitely generated
abelian groups are capable $($see \cite{ba} for more information$)$.
Also, among the non-abelian groups, the capability of  $p$-groups
took special attention, although the structure of all non-abelian
$p$-groups have not been characterized but the results of
\cite{bac,ma} is determined the capability of two generator
2-groups.

In the preset paper we are interesting to classify the capable group
where $G'$ is of order $p$ and $G/G'$ of rank $n-1$.

\section{preliminaries}
This section contains some definitions, theorems and lemmas which are used in main results.
 We assume that the notion of Schur multiplier is known, also we use the notion of epicenter
 and exterior center of a group without defining them. Epicenter of a group $G$ which is
  denoted by $Z^*(G)$, was introduced by  Beyl,  Felgner, and  Schmid in
  \cite{bey}.They showed a necessary and sufficient condition for a
  g roup to be capable is having trivial epicenter.

\begin{thm}$($See \cite[Theorem 2.5.10]{kar}$)$ \label{t1}
Let $G$ be a finite group and $N$ be a normal subgroup of $G$. Then
$N\subseteq Z^*(G)$ if and only if the natural map
$M(G)\longrightarrow M(G/N)$ is monomorphism.
\end{thm}

\begin{thm}$($See \cite[Theorem 2.2.10]{kar}$)$\label{t2}
Let $A$ and $B$ be finite groups then
\[\mathcal{M}(A\times B)\cong \mathcal{M}(A)\oplus \mathcal{M}(B)\oplus \frac{A}{A'}\otimes \frac{B}{B'}.\]
\end{thm}

\begin{thm}$($See \cite[Theorem 2.5.6 (i)]{kar}$)$\label{t3}
Let $G$ be a finite group and $N$ be a central subgroup of it, then
the following sequence is exact
\[\mathcal{M}(G)\longrightarrow \mathcal{M}(\frac{G}{N}) \longrightarrow N\cap G'\longrightarrow 1. \]
\end{thm}

The following lemma is a conclusion of Theorem \ref{t3} and used in
the proof of the main theorem.

\begin{lem}\label{l1}
Let $G$ be a finite $p$-group and $N\subseteq Z(G)\cap G'$ be a
subgroup of order $p$. If $|\mathcal{M}(G/N)|=p~|\mathcal{M}(G)|$
then $N\subseteq Z^*(G)$.
\end{lem}

\begin{proof}
Using Theorems \ref{t1} and \ref{t3} it is enough to show that $\mathcal{M}(G)\longrightarrow \mathcal{M}(G/N)$
has trivial kernel. Let $\alpha$ and $\beta$ denote the homomorphisms
$\mathcal{M}(G)\longrightarrow \mathcal{M}(G/N)$ and $\mathcal{M}(G/N) \longrightarrow N\cap G'$, respectively.
 Since $|N|=p$, we have $|ker \beta|=|\mathcal{M}(G/N)|/p$ which is equal
  to $|\mathcal{M}(G)|$. Now Theorem \ref{t3} implies $ker \alpha=1$ as required.
\end{proof}



The following lemma is a consequence of \cite[Corollary 2.5.3]{kar}.
\begin{lem}\label{t5}
Let $G$ be a finite $p$-group then
\[|\mathcal{M}(\frac{G}{\phi(G)})|\leq |\mathcal{M}(G)||\phi(G)\cap G'|\]
\end{lem}

\section{Main Results}
Let $G$ be a group of order $p^n$ and $G'$ is of order $p$ and
$G/G'$ is elementary abelian. By \cite[Lemma 2.1]{ni} we have
$G=H\cdot Z(G)$ in which $H$ is an extra special $p$-group
 and $``\cdot "$ denotes the central product of groups. We know that $G'\subseteq Z(G)$,
  now depending on the structure of $Z(G)$ and the way that $G'$ embeds in $Z(G)$ the
  structure of $G$ may be simplified as the following theorem asserts.

\begin{thm}\label{t6}
Let $|G|=p^n$ and $Z(G)$ is not cyclic then
\begin{itemize}
\item[(1)] if for some $K$, $Z(G)=G'\oplus K$ then $G=H\times K$;
\item[(2)] if $G'$ is not a direct summand of $Z(G)$ then
$G=\big(H\cdot \mathbb{Z}_{p^t}\big)\times K$ in which
$Z(G)=\mathbb{Z}_{p^{t+1}}\oplus K$ and $G'\subseteq
\mathbb{Z}_{p^{t+1}}$.
\end{itemize}
\end{thm}

\begin{proof}$(i)$ Since $G=H\cdot Z(G)$ and $H\cap Z(G)=G'$, so $G=H\times
K$.

$(ii)$ The proof is similar to the pervious part.
\end{proof}

The main theorem of this paper is

\begin{thm}\label{t7}
Let $G=H\cdot Z(G)$ be the group as above, then $G$ is capable if
and only if $H$ is capable and $G'$ is a direct summand of $Z(G)$.
\end{thm}

The proof of the Main Theorem is partitioned into some cases as
follows. From now on $G$ is of order $p^n$ and $G/G'$ is elementary
abelian of order $p^{n-1}$.

\begin{thm}\label{t8}
Let $G$ be as above and $Z(G)$ is cyclic then $G$ is not capable.
\end{thm}

\begin{proof}
Since $G/G'$ is elementary abelian $p$-group we have $\phi(G)=G'$.
Now using Lemma \ref{t5} and Main Theorems of \cite{ni,ni1}, we have
\[p^{\frac{1}{2}(n-1)(n-2)-1}\leq |\mathcal{M}(G)|\leq p^{\frac{1}{2}(n-1)(n-2)+1}.\]
Again using Main Theorems of \cite{ni,ni1}, we deduce that
\[|\mathcal{M}(G)|=p^{\frac{1}{2}(n-1)(n-2)-1},\]
so the following sequence is exact
\[1\longrightarrow \mathcal{M}(G)\longrightarrow \mathcal{M}(\frac{G}{G'})\longrightarrow G'\longrightarrow 1\]
which shows $G'\subseteq Z^*(G)$.
\end{proof}

\begin{thm}\label{t9}
Let $G$ be as above and $G'$ be a direct summand of $Z(G)$ then $G$ is capable if and only if $H$ is capable.
\end{thm}
\begin{proof}
Theorem \ref{t6} shows in this case $G=H\times K$ where $Z(G)=G'\oplus K$. Remember that we have
 $Z^{^{\wedge}}(G)\subseteq G'=H'$. Now depending on the capability of $H$ we have the following cases
\begin{itemize}
\item[(1)] $H$ is capable.
\item[(2)] $H$ is not capable.
\end{itemize}
    In case (1), for $p\neq 2$, $H$ is the extra special $p$-group of order $p^3$ and exponent $p$ and that $|\mathcal{M}(H)|=p^2$.
     If we show that $H'\nsubseteq Z^*(G)$ we have done. To do this we use Theorems \ref{t1} and \ref{t3}.
     The sequence
    \[\mathcal{M}(G)\longrightarrow \mathcal{M}(\frac{G}{H'})\longrightarrow H'\longrightarrow 1\]
    is exact, but using Theorem \ref{t2} we have $|\mathcal{M}(G)|=p~|\mathcal{M}(G/H')|$ so
    \[1\neq |ker(\mathcal{M}(G)\longrightarrow \mathcal{M}(G/H'))|\] and the result holds.
 For $p=2$, $G$ isomorphic to dihedral group of order $8$, and a similar technique shows the result.

 In case (2), $H$ can be either an extra special $p$-group of order $p^3$ and exponent $p^2$,
 or an extra special $p$-group of order $p^{2m+1}$ with $m>1$ which multipliers are trivial
 and of order $p^{2m^2-m-1}$, respectively. For $H$ of order $p^3$ a similar argument to
 that of the first case shows that
    \[\mathcal{M}(G)\longrightarrow \mathcal{M}(\frac{G}{H'})\]
    is injective and so $H'\subseteq Z^*(G)$. On the other hand
    if $H$ is of order $p^{2m+1}$ for $m>1$, using Theorem \ref{t2}
    the following sequence is exact, therefore $G$ is not capable.
    \[1\longrightarrow \mathcal{M}(G)\longrightarrow \mathcal{M}(\frac{G}{H'})\longrightarrow H'\longrightarrow 1\]

\end{proof}

Now the second case of Theorem \ref{t6} remains to discuss.

\begin{thm}\label{t10}
Let $G$ be as above with $Z(G)$ not cyclic and $G'$ is not a direct summand of $Z(G)$ then $G$ is not capable.
\end{thm}

\begin{proof}
In this case we have $G=H \cdot \mathbb{Z}_{p^t}\times K$ for some
$K\subseteq Z(G)$. Theorem \ref{t8} shows that $|\mathcal{M}(H\cdot
\mathbb{Z}_{p^t})|=p^{1/2(t-1)(t-2)-1}$ and also $Z^*(H\cdot
\mathbb{Z}_{p^t})=(HC_{p^t})'$ is of order $p$. With the aid of
Theorem \ref{t2}, we compute the orders of the Schur multipliers of
$G$ and $G/(H\cdot \mathbb{Z}_{p^t})'$, we have
\[|\mathcal{M}(\frac{G}{(H\cdot \mathbb{Z}_{p^t})'})|=p~|\mathcal{M}(G)|.\]
Now the results follows by Lemma \ref{l1}.
\end{proof}

\end{document}